\documentclass[11pt,a4paper,pdflatex,final]{amsart}
\pdfoutput=1
\usepackage{amsmath,amsthm,amssymb}
\usepackage{amsfonts}
\usepackage[top=25truemm,bottom=25truemm,left=30truemm,right=30truemm]{geometry}
\usepackage{newtxtext}
\usepackage[varg]{newtxmath}
\usepackage[mathscr]{eucal}
\usepackage{bxpapersize}
\usepackage{breqn}
\usepackage{ascmac}
\usepackage{epic,eepic}
\usepackage{graphicx}
\usepackage{subcaption}
\usepackage{multirow}
\usepackage{mathtools}
\mathtoolsset{showonlyrefs}
\usepackage{indentfirst}
\usepackage{cases}
\usepackage{color}
\usepackage{url}
\usepackage{bm}
\usepackage[all]{xy}
\usepackage[noadjust]{cite}
\usepackage{showkeys}
\usepackage{enumitem}
\usepackage{hyperref}
\hypersetup{
	setpagesize=false,
	bookmarksnumbered=true,%
	bookmarksopen=true,%
	colorlinks=true,%
	linkcolor=blue,
	citecolor=blue,
	urlcolor=magenta,
}

\theoremstyle{plain}

\newtheorem{theorem}{Theorem}[section]
\newtheorem{proposition}[theorem]{Proposition}
\newtheorem{lemma}[theorem]{Lemma}
\newtheorem{corollary}[theorem]{Corollary}
\newtheorem{fact}[theorem]{Fact}
\theoremstyle{definition}
\newtheorem{definition}[theorem]{Definition}
\newtheorem{example}[theorem]{Example}

\theoremstyle{remark}
\newtheorem{rmk}[theorem]{Remark}

\newcommand{\rank}{\operatorname{rank}}

\newcommand{\tr}{\operatorname{tr}}

\newcommand{\hess}{\operatorname{Hess}}

\newcommand{\sgn}{\operatorname{sgn}}

\newcommand{\R}{\mathbb{R}}
\newcommand{\bS}{\mathbb{S}}

\newcommand{\be}{\bm{e}}

\newcommand{\bv}{\bm{v}}
\newcommand{\first}{\operatorname{I}}
\newcommand{\second}{\operatorname{II}}

\newcommand{\mcal}{\mathcal}
\renewcommand{\phi}{\varphi}
\newcommand{\eps}{\varepsilon}
\newcommand{\what}{\widehat}

\newcommand{\til}{\tilde}

\newcommand{\inner}[2]{\left\langle{#1},{#2}\right\rangle}

\numberwithin{equation}{section}

\title[Cuspidal edges on focal surface]
{Cuspidal edges on focal surfaces of regular surfaces}
\author[K. Teramoto]{Keisuke Teramoto}
\address{Department of Mathematics, 
	Yamaguchi University, 1677-1 Yoshida, Yamaguchi 753-8512, Japan}
\email{kteramoto@yamaguchi-u.ac.jp}

\thanks{The author was partially supported by JSPS KAKENHI Grant Numbers JP22K13914 and JP22KK0034.}
\date{\today}

\begin{document}
	
	
	\begin{abstract}
	We investigate geometric invariants of cuspidal edges 
	on focal surfaces of regular surface. 
	In particular, we shall clarify the sign of the singular curvature 
	at a cuspidal edge on a focal surface using singularities of 
	parallel surface of a given surface satisfying certain conditions. 
	\end{abstract}
	
	\maketitle
	
\section{Introduction}
Let $f\colon U\to\R^3$ be a $C^\infty$ regular surface, 
where $U$ is an open set in $\R^2$. 
Then for a fixed $t\in\R$, a parallel surface 
$f^t\colon U\to\R^3$ of $f$ is defined by 
\[
f^t(u,v)=f(u,v)+t\nu(u,v)\quad  ((u,v)\in U),
\]
where $\nu$ is the unit normal vector field along $f$ 
given by $\nu=(f_u\times f_v)/|f_u\times f_v|$. 
A parallel surface is a (wave) front in $\R^3$ (\cite{agv,bruce,usy-book}). 
Let $\kappa_i$ ($i=1,2$) be principal curvatures of $f$. 
Then $f^t$ is singular at $p\in U$ when $t=1/\kappa_i(p)$ ($i=1$ or $2$). 
Singularities of parallel surfaces are Legendrian singularities. 
If $p$ is not an umbilic point of $f$, 
that is, $\kappa_1(p)\neq \kappa_2(p)$, 
then $f^t$ may be a cuspidal edge, a swallowtail, 
a cuspidal lips, a cuspidal beaks and a cuspidal butterfly at $p$ 
(see Definition \ref{dfn:singularities}; cf. \cite[Chapters 21 and 22]{agv}). 
Characterizations for these singularities are given by Fukui and Hasegawa 
\cite[Theorems 3.4 and 3.5]{fh-para} 
in terms of geometric properties of the initial surface. 
On the other hand, 
one can define focal surfaces $C_i\colon U\setminus\kappa_i^{-1}(0)\to\R^3$ ($i=1,2$) 
associated with $\kappa_i$ by 
\[
C_i(u,v)=f(u,v)+\frac{1}{\kappa_i(u,v)}\nu(u,v).
\]
The focal surface $C_i$ corresponds to the locus of 
centers of principal curvature spheres with respect to $\kappa_i$. 
The relationship between the Gaussian curvature 
of the focal surface and geometric properties of the initial surface 
is well understood (see \cite[page 183]{eisen}). 
Moreover, it is known that focal surfaces have Lagrangian singularities (\cite[Chapter 21]{agv}). 
Furthermore, it can be seen that the singularities of parallel surfaces 
are located on a focal surface. 
The relationships between several singularities of parallel surfaces 
and of focal surfaces are shown in Table \ref{tab:rel} 
(see Proposition \ref{prop:sing-focal}).		

\begin{table}[htbp]
	\caption{Relation between singularities.}
	\label{tab:rel}
				\begin{tabular}{|c||c|}\hline
						singularity of parallel surface& singularity of focal surface\\ \hline\hline
						 cuspidal edge & regular \\ \hline
						 swallowtail & {cuspidal edge}\\ \hline
					 cuspidal lips & cuspidal edge\\\hline
					cuspidal beaks & cuspidal edge\\ \hline
					\end{tabular}
			\end{table}

On the other hand, there are various studies on 
surface with cuspidal edge singularities 
from differential geometric viewpoint 
(cf. \cite{fukui,msuy,ms,tito,oset-tari,suy-front,tera0,tera1,tera3}). 
At a cuspidal edge, several geometric invariants are introduced (\cite{msuy,ms,suy-front}). 
The limiting normal curvature $\kappa_\nu$ and singular curvature $\kappa_s$ 
are known as principal geometric invariants at a cuspidal edge. 
The limiting normal curvature is related to the behavior 
of the Gaussian curvature near a cuspidal edge (\cite[Theorem A]{msuy}; see Figure \ref{fig:kn}). 
(We remark that the limiting normal curvature 
can be defied at a rank one singular point of a front (\cite[pages 253--254]{msuy}).) 
\begin{figure}[h]
	\centering
	\begin{minipage}[b]{0.3\hsize}
		\centering
		\includegraphics[width=0.7\linewidth]{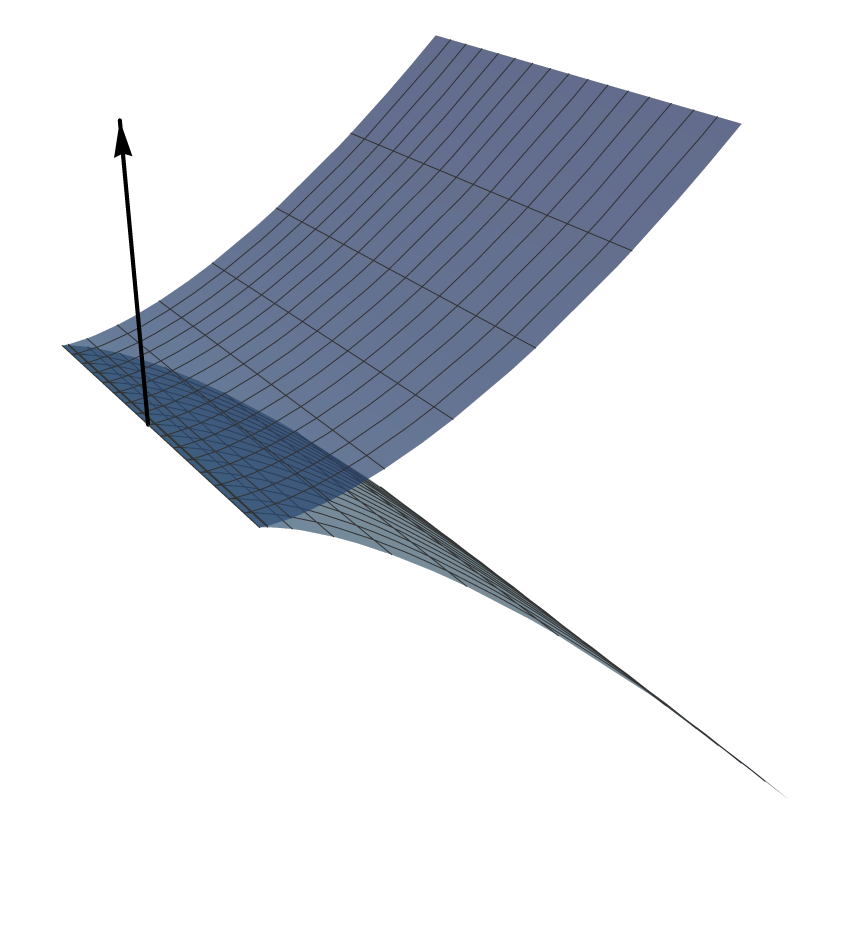}
		\subcaption{$\kappa_\nu=0$}
	\end{minipage}
	\begin{minipage}[b]{0.3\hsize}
		\centering
		\includegraphics[width=0.6\linewidth]{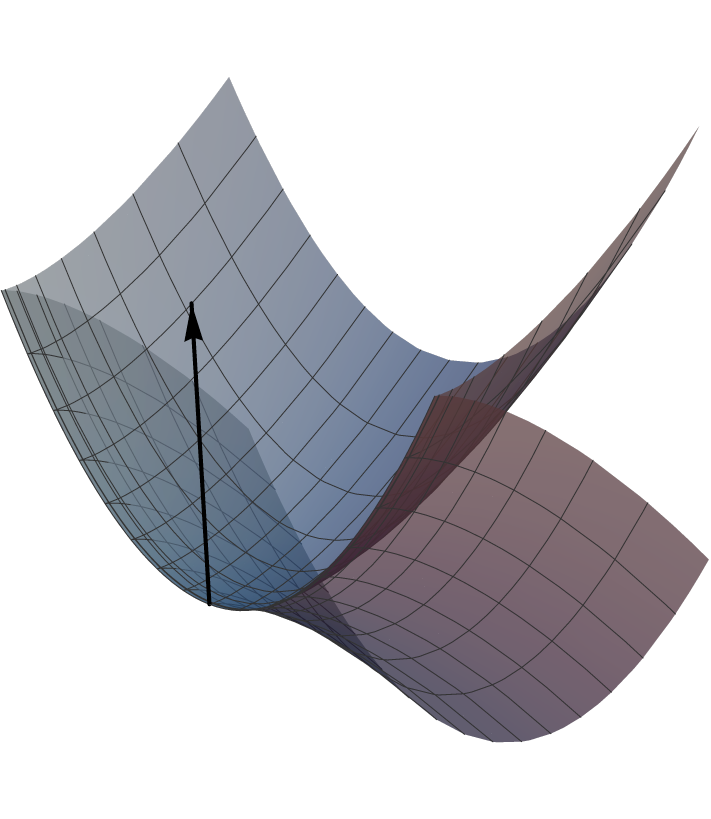}
		\subcaption{$\kappa_\nu\neq0$}
	\end{minipage}
	\caption{Cuspidal edges with 
		vanishing $\kappa_\nu$ (left) and 
		non-vanishing $\kappa_\nu$ (right). 
		Arrows indicate unit normal vectors.}
	\label{fig:kn}
\end{figure}
Conversely, the sign of the singular curvature reflects the concavity or convexity of the cuspidal edge 
(see Figure \ref{fig:ce}; \cite[Corollary 1.18]{suy-front}).
\begin{figure}[h]
	\centering
	\begin{minipage}[b]{0.3\hsize}
		\centering
		\includegraphics[width=0.65\linewidth]{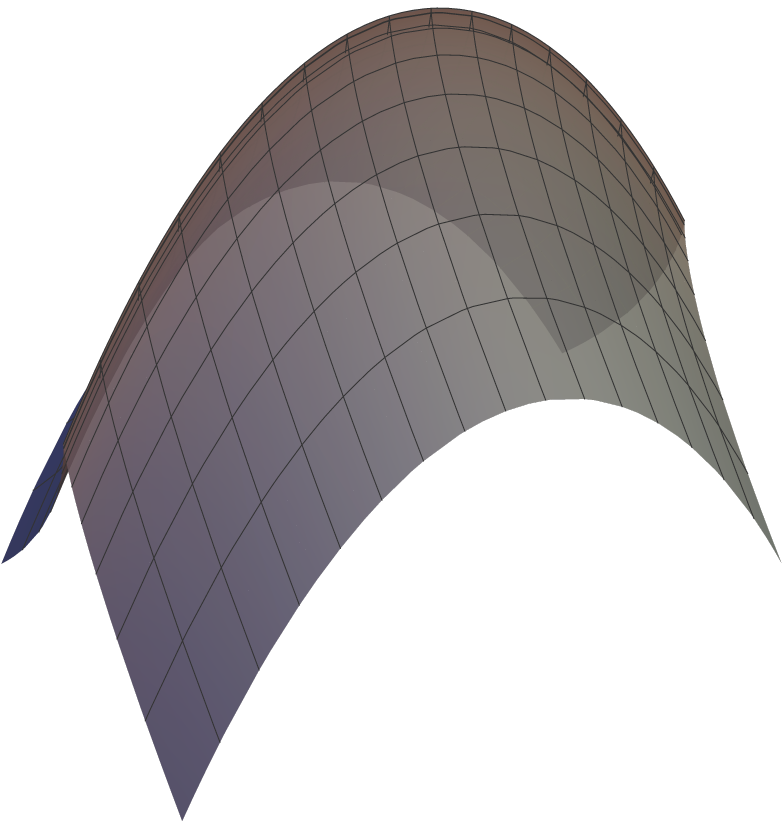}
		\subcaption{$\kappa_s>0$}
	\end{minipage}
	\begin{minipage}[b]{0.3\hsize}
		\centering
		\includegraphics[width=0.65\linewidth]{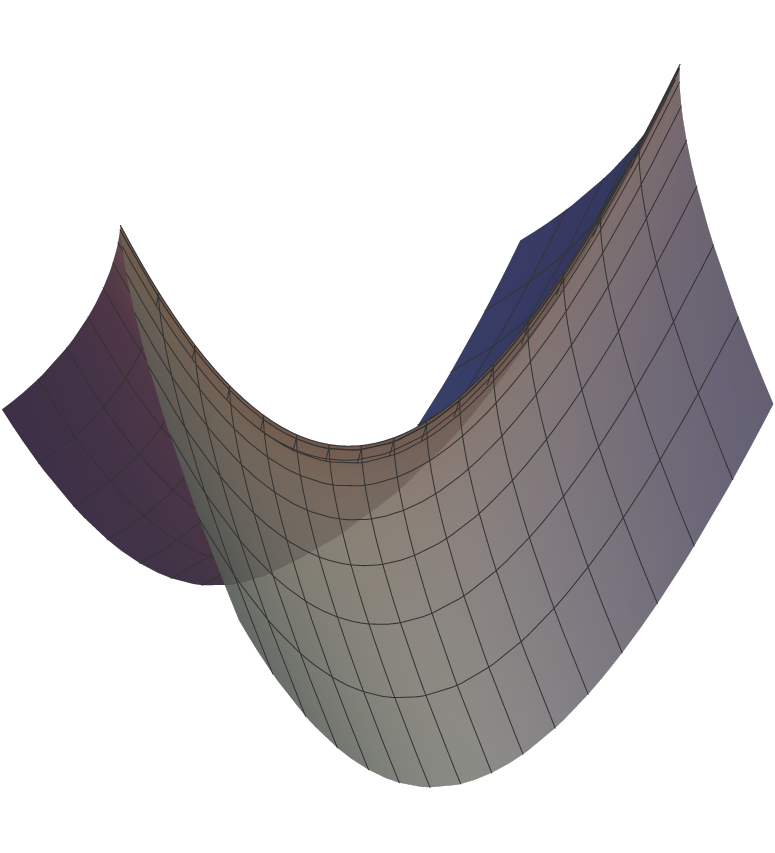}
		\subcaption{$\kappa_s<0$}
	\end{minipage}
	\caption{Cuspidal edges with 
		positive $\kappa_s$ (left) and 
		negative $\kappa_s$ (right).}
	\label{fig:ce}
\end{figure}
Therefore examining the singular curvature 
at a cuspidal edge on focal surfaces 
determines its concavity or convexity. 
Consequently, by investigating the properties of the initial surface or parallel surface 
that determines the sign of the singular curvature at cuspidal edges appearing on the focal surface, 
it is expected that the geometric properties of the focal surface will become clearer.

In this paper, we study geometric properties of parallel surfaces 
and focal surfaces at rank one singular points. 
In particular, we focus on relationships 
between geometric perspective and 
behavior of the limiting normal and singular curvatures. 
In Section \ref{sec:parallel}, 
we investigate parallel surfaces. 
When the initial surface does not have umbilic points, 
its parallel surfaces may have rank one singularities. 
Thus one can consider the limiting normal curvature at that point 
on a parallel surface. 
In particular, the following results hold.
\begin{proposition}\label{prop:kn-para}
	Let $f\colon U\to\R^3$ be a regular surface without umbilic point 
	and $\nu$ its unit normal vector. 
	Suppose that the principal curvature $\kappa_1$ \textup{(}resp. $\kappa_2$\textup{)}
	does not vanish on $U$. 
	{Then for the parallel surface $f^{t}$ with $t=1/\kappa_1(p)$ \textup{(}resp. $t=1/\kappa_2(p)$\textup{)},} 
	the limiting normal curvature 
	$\kappa_\nu^{t}$ of $f^t$ satisfies 
	\[
	\kappa_\nu^t=\frac{\kappa_1\kappa_2}{\kappa_1-\kappa_2}\quad 
	\left(\textit{resp.}~\kappa_\nu^t=\frac{\kappa_1\kappa_2}{\kappa_2-\kappa_1}\right)
	\]
	at $p$. 
\end{proposition}
The reciprocal $1/\kappa_\nu^t$ represents the (signed) distance between the two focal points 
(centers of two principal curvatures) at $p$. 
Indeed, we see that 
\[
\frac{1}{\kappa_\nu^t}=
\begin{dcases}
	\frac{1}{\kappa_2}-\frac{1}{\kappa_1} & (t=1/\kappa_1(p)),\\
	\frac{1}{\kappa_1}-\frac{1}{\kappa_2} & (t=1/\kappa_2(p))
\end{dcases}
\]
at $p$ when $\kappa_1(p)\kappa_2(p)\neq0$. 
On the other hand, it follows that 
\[
\kappa_\nu^t=
\begin{dcases}
	\kappa_2^t & (t=1/\kappa_1(p)),\\
	\kappa_1^t & (t=1/\kappa_2(p))
\end{dcases}
\]
hold at $p$, where $\kappa_i^t$ ($i=1,2$) are principal curvatures of $f^t$ (cf. \cite{tera1}). 
In addition, we examine the singular curvature of a cuspidal edge on $f^t$ 
(Proposition \ref{prop:ks-para}). 

In Section \ref{sec:focal}, 
we deal with focal surfaces. 
Especially, we study the limiting normal curvature 
and the singular curvature at a cuspidal edge on the focal surface. 
The curvature line coordinate system is employed to provide explicit representations 
of these curvatures. 
As a consequence, we establish a geometric criterion that determines the sign of the singular curvature of the focal surface in terms of the type of singularity of the parallel surface. 
More precisely, we shall show the following.
\begin{theorem}\label{thm:ks-c1}
	Let $f\colon U\to\R^3$ be a regular surface 
	whose parallel surface $f^{t}$ is a cuspidal lips 
	\textup{(}resp. a cuspidal beaks\textup{)} at $p\in U$, 
	where $t=1/\kappa_i(p)$ \textup{(}$i=1$ or $2$\textup{)}. 
	Suppose that the limiting normal curvature $\kappa_\nu^t$ of $f^t$ 
	vanishes at $p$.  
	Then the singular curvature $\kappa_s^{C_i}$ at a cuspidal edge $p$ of $C_i$ 
	is positive \textup{(}resp. negative\textup{)}. 
\end{theorem}
We note that if the assumption {$\kappa_\nu^t(p)=0$} is removed, 
the conclusion of Theorem \ref{thm:ks-c1} does not hold in general 
(Example \ref{ex:counter}).

\subsection*{Acknowledgments} 
The author thanks Kentaro Saji and Runa Shimada 
for fruitful discussions and valuable comments. 
The author is also grateful to the referees for 
careful reading and constructive comments, which improved the manuscript. 

\section{Preliminaries}
\subsection{Surfaces in Euclidean 3-space}
We here review some fundamental concepts from 
differential geometry of immersed surfaces 
in Euclidean $3$-space $\R^3$. 
For details, see \cite{carmo,hopf,porteous2}. 
{Let $U\subset\R^2$ be a domain. 
Let $f\colon U\to\R^3$ be an immersion 
and $\nu$ its unit normal vector field defined by 
\[
\nu(u,v)=\frac{f_u\times f_v}{|f_u\times f_v|}(u,v),
\] 
where $\times$ is the cross product of $\R^3$, 
$(\cdot)_u=\partial (\cdot)/\partial u$ and $(\cdot)_v=\partial (\cdot)/\partial v$. 
Then we call the image of $f$ or $f$ itself a (\textit{regular}) \textit{surface}. 
For a surface $f\colon U\to\R$, we define 
\begin{equation}\label{eq:fundamentals}
	\begin{aligned}
		E&=\inner{f_u}{f_u}, & F&=\inner{f_u}{f_v},& G&=\inner{f_v}{f_v},\\
		L&=-\inner{f_{u}}{\nu_u}, & M&=-\inner{f_u}{\nu_v}=-\inner{f_v}{\nu_u}, & 
		N&=-\inner{f_v}{\nu_v}, 
	\end{aligned}
\end{equation}
where $\inner{\cdot}{\cdot}$ is the Euclidean inner product. 
The functions $E,F,G$ are called the \textit{coefficients of the first fundamental form} of $f$, 
and $L,M,N$ are the \textit{coefficients of the second fundamental form} of $f$. 
Note that $EG-F^2>0$ holds on $U$. 

Let $\what{\first}$ and $\what{\second}$ be matrices given by 
\[
\what{\first}=\begin{pmatrix}
	E & F \\ F & G
\end{pmatrix},\quad 
\what{\second}=\begin{pmatrix}
	L & M \\ M & N
\end{pmatrix}.
\]
Then $\what{\first}$ is invertible. 
We define the Weingarten matrix $W$ as 
\[
W=\what{\first}^{-1}\what{\second}
=\frac{1}{EG-F^2}\begin{pmatrix}
	GL-FM & GM-FN \\ EM-FL & EN-FM
\end{pmatrix}.
\]
The \textit{Gaussian curvature} $K$ 
and \textit{mean curvature} $H$ of $f$ are defined by 
\[
K=\det{W},\quad H=\frac{1}{2}\tr W.
\]
Moreover, the eigenvalues $\kappa_1,\kappa_2$ of $W$ are called the \textit{principal curvatures} of $f$. 
A point $p\in U$ is called an \textit{umbilic point} of $f$ 
if $\kappa_1(p)=\kappa_2(p)$. 

Let $f\colon U\to\R^3$ be a surface and $\nu$ its unit normal vector field of $f$. 
Assume that $p\in U$ is not an umbilic point of $f$. 
Then there exist a sufficiently small neighborhood $\til{U}\subset U$ of $p$ 
and local coordinate system $(u,v)$ such that 
$F=M=0$ on $\til{U}$ (cf. \cite[page 188]{carmo}). 
We call such a coordinate system $(u,v)$ a \textit{curvature line coordinate system} on $\til{U}$. 

In the following, we assume that a surface $f\colon U\to\R^3$ has no umbilic point 
and parametrized by a curvature line coordinate system $(u,v)$. 
We define $\be_1,\be_2\colon U\to\R^3$ by 
\begin{equation}\label{eq:frame}
	\be_1=\frac{f_u}{|f_u|},\quad 
	\be_2=\frac{f_v}{|f_v|}.
\end{equation}
Then replacing $\nu$ with $-\nu$ if necessary 
it holds that 
\[
\inner{\be_i}{\be_j}=\delta_{ij},\quad 
\be_1\times \be_2=\nu,
\]
where $\delta_{ij}$ is the Kronecker delta. 
The Gauss and Weingarten formulas are as follows: 
\begin{equation}\label{eq:gauss}
	\begin{aligned}
		f_{uu}&=\Gamma^{1}_{11}f_u+\Gamma^{2}_{11}f_v+L\nu
		=\frac{E_u}{2E}f_u - \frac{E_v}{2G}f_v+L\nu,\\
		f_{uv}&=f_{vu}=\Gamma^{1}_{12}f_u+\Gamma^{2}_{12}f_v
		=\frac{E_v}{2E}f_u+\frac{G_u}{2G}f_v,\\
		f_{vv}&=\Gamma^{1}_{22}f_u+\Gamma^{2}_{22}f_v+N\nu
		=-\frac{G_u}{2E}f_u+\frac{G_v}{2G}f_v+N\nu,\\
		\nu_u&=-\kappa_1f_u,\quad 
		\nu_v=-\kappa_2f_v,
	\end{aligned}
\end{equation}
where $\kappa_1=L/E$ and $\kappa_2=N/G$. 
}
Moreover, Codazzi equations can be written as  
\begin{equation}\label{eq:codazzi}
L_v=\frac{1}{2}(\kappa_1+\kappa_2)E_v,\quad 
N_u=\frac{1}{2}(\kappa_1+\kappa_2)G_u.
\end{equation}

\begin{lemma}\label{lem:d-kappa}
Under the above setting, we have 
\[
\frac{E_v}{2E}=\Gamma^{1}_{12}=\frac{(\kappa_1)_{v}}{\kappa_2-\kappa_1},\quad 
\frac{G_u}{2G}=\Gamma^{2}_{12}=\frac{(\kappa_2)_{u}}{\kappa_1-\kappa_2}.
\]
\end{lemma}
\begin{proof}
We consider $\kappa_1$. 
By the definition, we have 
\[
(\kappa_1)_{v}
=\frac{L_v}{E}-\frac{LE_v}{E^2}
=\frac{L_v}{E}-\kappa_1\frac{E_v}{E}.
\]
By Codazzi equations \eqref{eq:codazzi}, it holds that 
\[
(\kappa_1)_{v}
=\frac{E_v}{2E}(\kappa_1+\kappa_2)-\kappa_1\frac{E_v}{E}
=\frac{E_v}{2E}(\kappa_2-\kappa_1).
\]
Thus we obtain the assertion for $\kappa_1$. 
For $\kappa_2$, one can show in a similar way. 
\end{proof}
By \eqref{eq:gauss} and Lemma \ref{lem:d-kappa}, 
it holds that 
\[
f_{uv}=-\frac{(\kappa_1)_{v}}{\kappa_1-\kappa_2}f_u+\frac{(\kappa_2)_{u}}{\kappa_1-\kappa_2}f_v. 
\] 

\subsection{Fronts}
To investigate parallel surfaces below, 
we here recall the class of surfaces with certain singularities 
called fronts (\cite{agv,usy-book}). 
Let $f\colon U\to\R^3$ be a $C^\infty$ map, 
where $U\subset \R^2$ is an open set. 
Then $f$ is said to be a \textit{front} 
if there exists a $C^\infty$ map $\nu\colon U\to\bS^2$ 
such that 
\begin{enumerate}
\item $\inner{df_q(\bv)}{\nu(q)}=0$ for any $q\in U$ and $\bv\in T_q\R^2$, and 
\item the pair $(f,\nu)\colon U\to\R^3\times \bS^2$ gives an immersion.
\end{enumerate}
We call the map $\nu$ a \textit{unit normal vector field} of $f$.
By definition, an immersion $f\colon U\to\R^3$ is a front. 

Let $f\colon U\to\R^3$ be a front with a unit normal vector field $\nu$. 
Then a point $p\in U$ is said to be a \textit{singular point} of $f$ 
if $f$ is not an immersion at $p$, that is, $\rank df_p<2$. 
{We denote the set of singular points of $f$ by $S(f)$.} 
We define a function $\lambda\colon U\to\R$ by 
\begin{equation}\label{eq:lambda}
\lambda(u,v)=\det(f_u,f_v,\nu)(u,v).
\end{equation}
We then call $\lambda$ the \textit{signed area density function} of $f$. 
It is clear that $S(f)=\lambda^{-1}(0)$ holds. 
Moreover, a function $\til{\lambda}$ which is non-zero proportional to $\lambda$ 
is called an \textit{identifier of singularities} of $f$. 
A singular point $p\in S(f)$ of $f$ is said to be of \textit{non-degenerate} 
if $(\tilde{\lambda}_u,\tilde{\lambda}_v)(p)\neq(0,0)$. 
Otherwise, it is said to be of \textit{degenerate}. 
For a non-degenerate singular point $p\in S(f)$, 
there exist a neighborhood $V(\subset U)$ of $p$ 
and a regular curve $\gamma\colon(-\eps,\eps)\to V$ with $\gamma(0)=p$ 
such that $\til{\lambda}(\gamma(t))=0$ for $t\in(-\eps,\eps)$. 
This curve $\gamma$ is called a \textit{singular curve} for $f$ 
passing through $p=\gamma(0)$. 
On the other hand, for a singular point $p\in S(f)$ 
satisfying $\rank df_p=1$, there exists 
a non-zero vector field $\eta$ on $V(\subset U)$ 
such that $\eta_p$ generates the kernel of $df_p$. 
This vector field $\eta$ is called a \textit{null vector field} for $f$. 

\begin{definition}
Let $f_i\colon(U_i,p_i)\to\R^3$ ($i=1,2$) be $C^\infty$ map germs. 
Then $f_1$ and $f_2$ are \textit{$\mcal{A}$-equivalent} 
if there exist diffeomorphism germs $\phi\colon(U_1,p_1)\to(U_2,p_2)$ 
and $\Phi\colon(\R^3,f_1(p_1))\to(\R^3,f_2(p_2))$ such that 
\(
\Phi\circ f_1=f_2\circ \phi
\).
\end{definition}
\begin{definition}\label{dfn:singularities}
Let $f\colon U\to\R^3$ be a front. 
Then 
\begin{enumerate}[label={\rm(\arabic*)}]
	\item $f$ is a \textit{cuspidal edge} at $p(\in U)$ 
	if $f$ is $\mcal{A}$-equivalent to the germ $(u,v)\mapsto(u,v^2,v^3)$ at the origin;
	\item $f$ is a \textit{swallowtail} at $p$ 
	if $f$ is $\mcal{A}$-equivalent to the germ 
	$(u,v)\mapsto(u,4v^3+2uv,3v^4+uv^2)$ at the origin;
	\item $f$ is a \textit{cuspidal butterfly} at $p$ 
	if $f$ is $\mcal{A}$-equivalent to the germ 
	$(u,v)\mapsto(u,5v^4+2uv,4v^5+uv^2)$ at the origin;
	\item $f$ is a \textit{cuspidal lips} at $p$ 
	if $f$ is $\mcal{A}$-equivalent to the germ 
	$(u,v)\mapsto(u,2v^3+u^2v,3v^4+u^2v^2)$ at the origin;
	\item $f$ is a \textit{cuspidal beaks} at $p$ 
	if $f$ is $\mcal{A}$-equivalent to the germ 
	$(u,v)\mapsto(u,2v^3-u^2v,3v^4-u^2v^2)$ at the origin.
\end{enumerate}
\end{definition}
These are rank one singularities of fronts. 
For these singularities, the following criteria are known 
(see \cite{krsuy,is-mandala,ist-horo,suy-ak,usy-book}). 
\begin{fact}\label{fact:crit}
Let $f\colon U\to\R^3$ be a front and $p\in U$ a rank one singular point of $f$. 
Then 
\begin{enumerate}[label={\rm(\arabic*)}]
	\item $f$ is a cuspidal edge at $p$ 
	if and only if $\eta\til{\lambda}(p)\neq0;$
	\item $f$ is a swallowtail at $p$ 
	if and only if $d\til{\lambda}(p)\neq0$, 
	$\eta\til{\lambda}(p)=0$ and $\eta^2\til{\lambda}(p)\neq0;$
	\item $f$ is a cuspidal butterfly at $p$ 
	if and only if $d\til{\lambda}(p)\neq0$, 
	$\eta\til{\lambda}(p)=\eta^2\til{\lambda}(p)=0$ 
	and $\eta^3\til{\lambda}(p)\neq0;$
	\item $f$ is a cuspidal lips at $p$ 
	if and only if $d\til{\lambda}(p)=0$ and $\det\hess(\til{\lambda})(p)>0;$
	\item $f$ is a cuspidal beaks at $p$ 
	if and only if $d\til{\lambda}(p)=0$, $\eta^2\til{\lambda}(p)\neq0$ and 
	$\det\hess(\til{\lambda})(p)<0$.
\end{enumerate}
Here $\til{\lambda}$ is an identifier of singularities of $f$, 
$\eta$ a null vector field, $\eta^k\til{\lambda}$ is 
the $k$-th order directional derivative of $\til{\lambda}$ 
by $\eta$, and $\det\hess(\til{\lambda})(p)$ is the Hessian of $\til{\lambda}$ at $p$. 
\end{fact}

\subsubsection{Limiting normal curvature and singular curvature} 
Let $f\colon U\to\R^3$ be a front and $\nu$ its unit normal vector {field}. 
Assume that $f$ is a cuspidal edge at $p$. 
Then one can find a singular curve $\gamma(t)$ passing through $p=\gamma(0)$ 
and a null vector field $\eta$ around $p$. 
We set $\hat{\gamma}=f\circ \gamma$. 
Then along $\gamma$, we can define the following geometric invariants: 
\begin{equation}\label{eq:invariants}
\begin{aligned}
	\kappa_\nu(t)&=\frac{\inner{\hat{\gamma}''(t)}{\nu(\gamma(t))}}{|\hat{\gamma}'(t)|^2}
	=-\frac{\inner{\hat{\gamma}'(t)}{\nu'(\gamma(t))}}{|\hat{\gamma}'(t)|^2}\quad 
	{\left('=\frac{d}{dt}\right)},\\ 
	\kappa_s(t)&=\sgn(\eta\lambda\cdot\det(\gamma',\eta))
	\frac{\det(\hat{\gamma}'(t),\hat{\gamma}''(t),\nu(\gamma(t)))}{|\hat{\gamma}'(t)|^3},
\end{aligned}
\end{equation}
where $\lambda$ is the signed area density function. 
We call $\kappa_\nu$ and $\kappa_s$ the \textit{limiting normal curvature} 
and the \textit{singular curvature}, respectively (see \cite{ms,msuy,suy-front,usy-book}). 


{
Let $f\colon U\to\R^3$ be a front. 
Let $p\in U$ be a rank one singular point of $f$. 
Assume that $p$ is a non-degenerate singular point other than a cuspidal edge. 
Then the limiting normal curvature $\kappa_\nu$ of $f$ at $p$ 
is given by 
\begin{equation}\label{eq:kn-nondeg}
	\kappa_\nu(p)=\lim_{t\to0}\frac{\inner{\hat{\gamma}''(t)}{\nu(\gamma(t))}}{|\hat{\gamma}'(t)|^2}
	=-\lim_{t\to0}\frac{\inner{\hat{\gamma}'(t)}{\hat{\nu}'(t)}}{|\hat{\gamma}'(t)|^2},
\end{equation}
where $\hat{\gamma}(t)=f(\gamma(t))$, $\hat{\nu}(t)=\nu(\gamma(t))$ 
and $\gamma(t)$ is the singular curve of $f$ through $p=\gamma(0)$ 
(\cite[Proposition 2.9]{msuy}). 
On the other hand, assume that $p$ is a degenerate singular point. 
Take a local coordinate system $(u,v)$ satisfying $f_v(p)=0$. 
Then the limiting normal curvature $\kappa_\nu(p)$ of $f$ at $p$ is defined by 
\begin{equation}\label{eq:kn-deg}
	\kappa_\nu(p)=\frac{\inner{f_{uu}(p)}{\nu(p)}}{|f_u(p)|^2}
\end{equation}
(\cite[(2.2)]{msuy}). 
In this case, $\kappa_\nu(p)$ does not depend on the choice 
of local coordinates satisfying $f_v(p)=0$ (\cite[Proposition 2.9]{msuy}). 
}

\section{Parallel surfaces}\label{sec:parallel}
We consider the parallel surfaces of a regular surface $f\colon U\to\R^3$ 
without umbilic point. 
For a fixed $t\in\R\setminus\{0\}$, 
a \textit{parallel surface} $f^t$ of $f$ at distance $t$ 
is defined as 
\begin{equation}\label{eq:parallel}
f^t(u,v)\coloneqq f(u,v)+t\nu(u,v),
\end{equation}
where $\nu$ is a unit normal vector to $f$. 
Let us take curvature line coordinates $(u,v)$. 
Then we see that 
\begin{equation}\label{eq:d-ft}
f^t_u=(1-t\kappa_1)f_u,\quad 
f^t_v=(1-t\kappa_2)f_v
\end{equation}
by Weingarten formulas {\eqref{eq:gauss}}.  
Thus $\nu$ can be also a unit normal vector to $f^t$. 
Moreover, $f^t$ is a front for any $t\in\R\setminus\{0\}$ (\cite{usy-book}). 
%

For each fixed $t\in\R\setminus\{0\}$, 
the set of singular points $S(f^t)$ of $f^t$ is 
\[
S(f^t)=\bigcup_{i=1}^2\{(u,v)\in U\mid \kappa_i(u,v)=1/t\}.
\]
Since $f$ has no umbilic points, 
\[
\bigcap_{i=1}^2\{(u,v)\in U\mid \kappa_i(u,v)=1/t\}=\emptyset.
\]
If we take $t=1/\kappa_1(p)$, 
$S(f^t)$ coincides with $\kappa_1^{-1}(\kappa_1(p))$. 
This implies that the \textit{constant principal curvature lines} form 
the set of singular points of a parallel surface (see \cite{fh-para}). 

\subsection{Singularities of parallel surfaces}
In the following of this section, we fix $t=1/\kappa_1(p)$. 
Then we see that 
$f^t_u(p)=0$ and $f^t_v(p)\neq0$ by \eqref{eq:d-ft}. 
Thus a point $p$ is a corank one singular point of $f^t$. 
Moreover, we set 
\begin{equation}\label{eq:cpc}
\til{\lambda}(u,v)=\kappa_1(u,v)-\kappa_1(p).
\end{equation}
Then $\til{\lambda}^{-1}(0)=S(f^t)$. 
Thus $\til{\lambda}$ is an identifier of singularities of $f^t$. 
In addition, $\partial_u$ give a null vector field for $f^t$.
By Fact \ref{fact:crit}, we have the following immediately. 
\begin{proposition}[{\cite[Theorems 3.4 and 3.5]{fh-para}}]\label{prop:lips/beaks}
Let $f\colon U\to\R^3$ be a regular surface without umbilic point 
and $f^t$ the parallel surface of $f$ at distance $t=1/\kappa_1(p)$. 
Take a curvature line coordinate system $(u,v)$ on $U$ for $f$. 
\begin{enumerate}[label={\rm(\arabic*)}]
	\item $f^t$ is a cuspidal edge at $p$ 
	if and only if $(\kappa_1)_u\neq0$ at $p$.
	\item $f^t$ is a swallowtail at $p$ 
	if and only if $(\kappa_1)_u=0$, $(\kappa_1)_v\neq0$ and $(\kappa_1)_{uu}\neq0$ at $p$.
	\item $f^t$ is a cuspidal butterfly at $p$ 
	if and only if 
	$(\kappa_1)_u=(\kappa_1)_{uu}=0$, $(\kappa_1)_v\neq0$ and $(\kappa_1)_{uuu}\neq0$ at $p$.
	\item $f^t$ is a cuspidal lips at $p$ 
	if and only if $(\kappa_1)_u=(\kappa_1)_v=0$ and 
	$\det\hess(\kappa_1)>0$ at $p$.
	\item $f^t$ is a cuspidal beaks at $p$ 
	if and only if $(\kappa_1)_u=(\kappa_1)_v=0$, 
	$(\kappa_1)_{uu}\neq0$ and
	$\det\hess(\kappa_1)<0$ at $p$.
\end{enumerate} 
\end{proposition}
{ 
We remark that if $f^t$ ($t=1/\kappa_1(p)$) is a cuspidal lips, 
then $(\kappa_1)_{uu}(p)\neq0$. 
If we consider the parallel surface $f^t$ with $t=1/\kappa_2(p)$, 
$\til{\lambda}=\kappa_2-\kappa_2(p)$ is an identifier of singularities 
and $\partial_v$ is a null vector field for $f^t$. 
Thus one can obtain similar characterizations as Proposition \ref{prop:lips/beaks}.}

\subsection{The limiting normal curvature}
{
We here consider the limiting normal curvature $\kappa_\nu^t$ of $f^t$ 
at a rank one singular point.
In particular, we give the proof of Proposition \ref{prop:kn-para}. 
}
\begin{proof}[Proof of Proposition \ref{prop:kn-para}]
We take a curvature line coordinates $(u,v)$. 
{We show the case where $t=1/\kappa_1(p)$.} 
We first consider the case where the parallel surface $f^t$ is a cuspidal edge at $p$. 
An identifier of singularities $\til{\lambda}$ for $f^t$ 
is given by \eqref{eq:cpc}. 
By Proposition \ref{prop:lips/beaks}, 
there exists a function $u(v)$ such that 
$\til\lambda(u(v),v)=0$. 
We set a curve $c(v)=(u(v),v)$, and 
define $\hat{c}(v)=f^t(u(v),v)$. 
Direct calculations show 
\begin{equation}\label{eq:diff-c}
	\begin{aligned}
		\hat{c}'(v)&=f^t_u(c(v))u'(v)+f^t_v(c(v))
		=\left(1-\frac{\kappa_2(c(v))}{\kappa_1(c(v))}\right)f_v(c(t)),\\
		\hat{\nu}'(v)&=-\kappa_1(c(v))f_u(c(v))u'(v)-\kappa_2(c(v))f_v(c(v)),
	\end{aligned}
\end{equation}
where $'=d/dv$ and $\hat{\nu}(v)=\nu(u(v),v)$. 
{By \eqref{eq:invariants},} we have 
\[
\kappa_\nu^{t}(v)=-\frac{\inner{\hat{c}'(v)}{\hat{\nu}'(v)}}{|\hat{c}'(v)|^2}
=\frac{\kappa_1(c(v))\kappa_2(c(v))}{\kappa_1(c(v))-\kappa_2(c(v))}.
\]	

We next show the case where $p$ is a non-degenerate singular point of $f^t$ 
other than a cuspidal edge. 
In this case, we get $(\kappa_1)_u(p)=(\kappa_1)_{uu}(p)=0$ 
and $(\kappa_1)_{v}(p)\neq0$. 
Thus by the implicit function theorem, 
there exists a function $v(u)$ such that $\kappa_1(u,v(u))-\kappa_1(p)=0$. 
Let us set $c(u)=(u,v(u))$ and $\hat{c}(u)=f^t(u,v(u))$. 
Then it holds that 
\[
\begin{aligned}
	\dot{\hat{c}}(u)&=f^t_u(c(u))+f^t_{v}(c(u))\dot{v}(u)
	=\left(1-\frac{\kappa_2(c(u))}{\kappa_1(c(u))}\right)f_v(c(u))\dot{v}(u),\\
	\dot{\hat{\nu}}(u)&=
	-\kappa_1(c(u))f_u(c(u))-\kappa_2(c(u))f_v(c(u))\dot{v}(u),
\end{aligned}
\]
where $\hat{\nu}(u)=\nu(u,v(u))$ and $\dot{}=d/du$. 
Therefore we get 
\[
\begin{aligned}
	|\dot{\hat{c}}(u)|^2
	&=\frac{(\kappa_1(c(u))-\kappa_2(c(u)))^2}{\kappa_1(c(u))^2}\dot{v}(u)^2G(c(u)),\\
	-\inner{\dot{\hat{c}}(u)}{\dot{\hat{\nu}}(u)}
	&=\frac{\kappa_2(c(u))(\kappa_1(c(u))-\kappa_2(c(u)))}{\kappa_1(c(u))}\dot{v}(u)^2G(c(u)).
\end{aligned}
\]
Hence {by \eqref{eq:kn-nondeg},} we have 
\[
\kappa_\nu^t(p)
=\lim_{u\to u_0}-\frac{\inner{\dot{\hat{c}}(u)}{\dot{\hat{\nu}}(u)}}{|\dot{\hat{c}}(u)|^2}
=\frac{\kappa_1(p)\kappa_2(p)}{\kappa_1(p)-\kappa_2(p)}.
\]

We finally deal with the case where $p$ is a degenerate corank one singular point of $f^t$. 
Since $f^t_u(p)=0$ and $f^t_v(p)\neq0$, 
the limiting normal curvature can be calculated as 
\[
\kappa_\nu^t(p)=\frac{\inner{f^t_{vv}}{\nu(p)}}{|f^t_v(p)|^2}
\]
{by \eqref{eq:kn-deg}}. 
By \eqref{eq:d-ft}, we have 
\[
f^t_{vv}=-t(\kappa_2)_vf_v+(1-t\kappa_2)f_{vv}.
\]
Thus we get 
\[
|f^t_v(p)|^2=\frac{(\kappa_1(p)-\kappa_2(p))^2}{\kappa_1(p)^2}G(p),\quad 
\inner{f^t_{vv}(p)}{\nu(p)}
=\frac{\kappa_2(p)(\kappa_1(p)-\kappa_2(p))}{\kappa_1(p)}G(p),
\]
where we used $N(p)=\kappa_2(p)G(p)$. 
Therefore $\kappa_\nu^t(p)$ can be written by 
\[
\kappa_\nu^t(p)=\frac{\kappa_1(p)\kappa_2(p)}{\kappa_1(p)-\kappa_2(p)}.
\]
Hence we have the assertions. 
{For the case of $t=1/\kappa_2(p)$, one can show in a similar way.}
\end{proof} 

\subsection{The singular curvature}
We here consider the singular curvature of the cuspidal edge 
appeared as a singularity of a parallel surface. 
\begin{proposition}\label{prop:ks-para}
	Let $f\colon U\to\R^3$ be a regular surface without umbilic point. 
	Assume that $f$ is parametrized by the curvature line coordinate system $(u,v)$. 
	If the parallel surface $f^t$ with $t=1/\kappa_1(p)$ \textup{(}resp. $t=1/\kappa_2(p)$\textup{)} 
	is a cuspidal edge at $p\in U$, then the singular curvature of $f^t$ at $p$ is  
	\[
	\kappa_s^t=-\frac{|\kappa_1|((\kappa_1)_v^2E-(\kappa_1)_u(\kappa_2)_uG)}
	{|(\kappa_1)_u|(\kappa_1-\kappa_2)^2G\sqrt{E}}\quad 
	\left(\text{resp.}~\kappa_s^t=-\frac{|\kappa_2|((\kappa_2)_u^2G-(\kappa_1)_v(\kappa_2)_vE)}
	{|(\kappa_2)_v|(\kappa_1-\kappa_2)^2E\sqrt{G}}\right).
	\]
\end{proposition}
\begin{proof}
	We show the case that $f^t$ with $t=1/\kappa_1(p)$ 
	is a cuspidal edge at $p$. 
	Then by the previous arguments, there exists a regular curve $c(v)=(u(v),v)$ 
	passing through $p=c(v_0)$ such that $\til{\lambda}(c(v))=0$, 
	where $\til{\lambda}(u,v)=\kappa_1(u,v)-\kappa_1(p)$. 
	We set $\hat{c}(v)=f^t(c(v))$. 
	Then by \eqref{eq:diff-c}, the second order differential $\hat{c}''(v)$ can be calculated as 
	\[
	\hat{c}''(v)=-\left(\frac{\kappa_2(c(v))}{\kappa_1(c(v))}\right)'f_v(c(v))
	+\left(1-\frac{\kappa_2(c(v))}{\kappa_1(c(v))}\right)(f_{uv}(c(v))u'(v)+f_{vv}(c(v))).
	\]
	Since $\til{\lambda}(c(v))=\kappa_1(c(v))-\kappa_1(p)=0$, we have 
	\(
	(\kappa_1)_u(c(v))u'(v)+(\kappa_1)_v(c(v))=0
	\). 
	Since $(\kappa_1)_u(p)\neq0$, it holds that 
	\begin{equation}\label{eq:du2}
		u'(v_0)=-\frac{(\kappa_1)_v(p)}{(\kappa_1)_u(p)}.
	\end{equation}
	By \eqref{eq:gauss}, Lemma \ref{lem:d-kappa} and \eqref{eq:du2}, 
	we get
	\[
	\begin{aligned}
		f_{uv}u'+f_{vv}
		&=-\frac{(\kappa_1)_v}{(\kappa_1)_u}
		\left(-\frac{(\kappa_1)_v}{\kappa_1-\kappa_2}f_u
		+\frac{(\kappa_2)_u}{\kappa_1-\kappa_2}f_v\right)
		-\frac{G_u}{2E}f_u+\frac{G_v}{2G}f_v+N\nu\\
		&=\frac{(\kappa_1)_v^2E-(\kappa_1)_u(\kappa_2)_uG}{(\kappa_1)_u(\kappa_1-\kappa_2)E}f_u
		+\left(\frac{G_v}{2G}-\frac{(\kappa_1)_v(\kappa_2)_u}{\kappa_1-\kappa_2}\right)f_v+N\nu
	\end{aligned}
	\]
	at $p$. 
	Therefore we have 
	\begin{equation}\label{eq:diff-c2}
		\hat{c}''=\frac{(\kappa_1)_v^2E-(\kappa_1)_u(\kappa_2)_uG}{\kappa_1(\kappa_1)_uE}f_u
		+Af_v+N\nu
	\end{equation}
	at $p=c(v_0)$, where $A$ is some constant. 
	By \eqref{eq:diff-c} and \eqref{eq:diff-c2}, it holds that 
	\begin{equation}\label{eq:det-para}
		\begin{aligned}
		\det(\hat{c}',\hat{c}'',\nu)
		&=\frac{(\kappa_1-\kappa_2)}{\kappa_1}\frac{(\kappa_1)_v^2E-(\kappa_1)_u(\kappa_2)_uG}{\kappa_1(\kappa_1)_uE}
		\det(f_v,f_u,\nu)\\
		&=-\frac{(\kappa_1-\kappa_2)((\kappa_1)_v^2E-(\kappa_1)_u(\kappa_2)_uG)}
		{\kappa_1^2(\kappa_1)_u}\sqrt{\frac{G}{E}}
		\end{aligned}
	\end{equation}
	at $p$. 
	On the other hand, the signed area density function $\lambda^t$ of $f^t$ is 
	\[
	\lambda^t=(1-t\kappa_1)(1-t\kappa_2)\det(f_u,f_v,\nu)
	=(1-t\kappa_1)(1-t\kappa_2)\sqrt{EG}.
	\]
	Since $\eta^t=\partial_u$ is a null vector field for $f^t$, 
	we see that 
	\[
	\eta^t\lambda^t=-t(\kappa_1)_u(1-t\kappa_2)\sqrt{EG}
	=-\frac{(\kappa_1)_u(\kappa_1-\kappa_2)}{\kappa_1^2}\sqrt{EG}
	\]
	holds at $p$. 
	Moreover, we obtain $\det(c',\eta^t)=-1$. 
	Hence we get 
	\[
	\sgn(\eta^t\lambda^t\det(c',\eta^t))
	=\sgn\left(\frac{(\kappa_1)_u(\kappa_1-\kappa_2)}{\kappa_1^2}\sqrt{EG}\right)
	=\sgn((\kappa_1)_u(\kappa_1-\kappa_2))
	\]
	at $p$. 
	Thus the singular curvature $\kappa_s^t$ is calculated as 
	\[
	\begin{aligned}
		\kappa_s^t&=-\sgn((\kappa_1)_u(\kappa_1-\kappa_2))
		\frac{\kappa_1^2|\kappa_1|}{(\kappa_1-\kappa_2)^2|\kappa_1-\kappa_2|G\sqrt{G}}
		\frac{(\kappa_1-\kappa_2)((\kappa_1)_v^2E-(\kappa_1)_u(\kappa_2)_uG)}
		{\kappa_1^2(\kappa_1)_u}\sqrt{\frac{G}{E}}\\
		&=-\sgn((\kappa_1)_u(\kappa_1-\kappa_2))
		\frac{|\kappa_1|((\kappa_1)_v^2E-(\kappa_1)_u(\kappa_2)_uG)}
		{(\kappa_1)_u(\kappa_1-\kappa_2)|\kappa_1-\kappa_2|G\sqrt{E}}\\
		&=-\frac{(\kappa_1)_u(\kappa_1-\kappa_2)}{|(\kappa_1)_u||\kappa_1-\kappa_2|}
		\frac{|\kappa_1|((\kappa_1)_v^2E-(\kappa_1)_u(\kappa_2)_uG)}
		{(\kappa_1)_u(\kappa_1-\kappa_2)|\kappa_1-\kappa_2|G\sqrt{E}}\\
		&=-\frac{|\kappa_1|((\kappa_1)_v^2E-(\kappa_1)_u(\kappa_2)_uG)}
		{|(\kappa_1)_u|(\kappa_1-\kappa_2)^2G\sqrt{E}}
	\end{aligned}
	\]
	at $p$. 
	For the case of $t=1/\kappa_2(p)$, 
	one can show in a similar way.
\end{proof}

\section{Focal surfaces}\label{sec:focal}
Let $f\colon U\to\R^3$ be a regular surface without umbilic point 
parametrized by a curvature line coordinate system $(u,v)$. 
Let $\kappa_i$ $(i=1,2)$ be principal curvatures for $f$. 
Then the maps 
\begin{equation}\label{eq:caustic}
C_i\coloneqq f+\frac{1}{\kappa_i}\nu\colon U\setminus\kappa_i^{-1}(0)\to\R^3 
\quad (i=1,2)
\end{equation}
are called \textit{focal surfaces} (or \textit{caustics}) of $f$. 
By direct calculations, we see that 
\begin{equation}\label{eq:diff-C}
\begin{aligned}
	(C_1)_u&=-\frac{(\kappa_1)_u}{\kappa_1^2}\nu,\quad 
	(C_1)_v=\left(1-\frac{\kappa_2}{\kappa_1}\right)f_v-\frac{(\kappa_1)_{v}}{\kappa_1^2}\nu,\\
	(C_2)_u&=\left(1-\frac{\kappa_1}{\kappa_2}\right)f_u-\frac{(\kappa_2)_{u}}{\kappa_2^2}\nu,\quad 
	(C_2)_v=-\frac{(\kappa_2)_{v}}{\kappa_2^2}\nu.
\end{aligned}
\end{equation}
Therefore $\be_1$ (resp. $\be_2$) as in \eqref{eq:frame} 
can be taken as a unit normal vector to $C_1$ (resp. $C_2$). 
{It is known that $C_1$ (resp. $C_2$) is a front (\cite[Theorem 7.1.1]{usy-book}).} 
Moreover, the signed area density functions for $C_1$ and $C_2$ are 
\begin{equation}\label{eq:lam-C}
\begin{aligned}
	\lambda^{C_1}&=\det((C_1)_u,(C_1)_v,\be_1)
	=\frac{(\kappa_1)_u}{\kappa_1^2}\left(1-\frac{\kappa_2}{\kappa_1}\right)|f_v|,\\ 
	\lambda^{C_2}&=\det((C_2)_u,(C_2)_v,\be_2)
	=\frac{(\kappa_2)_v}{\kappa_2^2}\left(1-\frac{\kappa_1}{\kappa_2}\right)|f_u|,
\end{aligned}
\end{equation}
respectively. 
Since $\kappa_1\neq \kappa_2$ on $U$, 
the sets of singular points of $C_i$ ($i=1,2$) are 
\[
S(C_1)=\{q\in U\mid (\kappa_1)_u(q)=0\},\quad 
S(C_2)=\{q\in U\mid (\kappa_2)_v(q)=0\}
\]
(cf. \cite{ifrt,porteous1}). 
For a point $p\in S(C_1)$ (resp. $p\in S(C_2)$), 
we see that $\rank d(C_1)_p=1$ (resp. $\rank d(C_2)_p=1$). 
Thus one can take 
$\eta^{C_1}=\partial_u$ (resp. $\eta^{C_2}=\partial_v$) 
as a null vector field for $C_1$ (resp. $C_2$). 
{By Fact \ref{fact:crit}, we have the following.

\begin{proposition}[cf. {\cite[Theorem 7.1.1]{usy-book}}]\label{prop:sing-C}
	Let $f\colon U\to\R^3$ be a surface without umbilic point 
	parametrized by a curvature line coordinate system $(u,v)$. 
	Suppose that $\kappa_1$ \textup{(}resp. $\kappa_2$\textup{)} 
	does not vanish on $U$. 
	Then the focal surface $C_1$ \textup{(}resp. $C_2$\textup{)}
	is a cuspidl edge at $p\in U$ 
	if and only if $(\kappa_1)_u(p)=0$ and $(\kappa_1)_{uu}(p)\neq0$ 
	\textup{(}resp. $(\kappa_2)_v(p)=0$ and $(\kappa_2)_{vv}(p)\neq0$\textup{)}. 
\end{proposition}
By Propositions \ref{prop:lips/beaks} and \ref{prop:sing-focal}, 
the following assertions hold. 
}
\begin{proposition}[cf. \cite{tera2}]\label{prop:sing-focal}
\begin{enumerate}[label={\rm(\arabic*)}]
	\item When the parallel surface $f^t$ with $t=1/\kappa_i(p)$ 
	is a cuspidal edge at $p$, then $C_i$ is regular at $p$. 
	\item When $f^t$ is either a swallowtail, a cuspidal lips 
	or a cuspidal beaks at $p$, then $C_i$ is a cuspidal edge at $p$. 
\end{enumerate} 
\end{proposition}

\begin{proof}
We show the case of $C_1$. 
Let us take a curvature line coordinate system $(u,v)$. 
{The first assertion holds by the previous discussion.} 
{Thus we show the second assertion.} 
We assume that $f^t$ is one of a swallowtail, a cuspidal lips or a cuspidal beaks. 
Then by Proposition \ref{prop:lips/beaks}, it holds that 
{$(\kappa_1)_u(p)=0$ and $(\kappa_1)_{uu}(p)\neq0$}. 
{Thus by Proposition \ref{prop:sing-C}, 
	the assertion follows.}

For the case of $C_2$, one can show in a similar way. 
\end{proof}

\begin{rmk}
This proposition implies that 
if $f$ is a front with a cuspidal lips or cuspidal beaks, 
then its (suitable) focal surface is a cuspidal edge 
at the corresponding point (cf. \cite{tera2}). 
\end{rmk}

When the parallel surface $f^t$ with $t=1/\kappa_i(p)$ ($i=1$ or $2$) is a cuspidal edge at $p$, 
the Gaussian curvature of $C_i$ and the singular curvature of $f^t$ are related as follows: 
\begin{corollary}
	Let $f\colon U\to\R^3$ be a regular surface. 
	Assume that the parallel surface $f^t$ with $t=1/\kappa_i(p)$ 
	\textup{(}$i=1$ or $2$\textup{)} of $f$ 
	is a cuspidal edge at $p$. 
	If the Gaussian curvature of $C_i$ of the focal surface $C_i$ vanishes at $p$, 
	then the singular curvature of $f^t$ is non-positive at $p$. 
\end{corollary}
\begin{proof}
	Take a curvature line coordinate system $(u,v)$. 
	We then show the case for $i=1$. 
	In this case, the Gaussian curvature of $C_1$ vanishes at $p$ 
	if and only if $(\kappa_2)_u(p)=0$ by \eqref{eq:KC} in the Appendix \ref{sec:appendix}. 
	In this case, the singular curvature $\kappa_s^t$ of $f^t$ at $p$ is 
	\[
	\kappa_s^t=-\frac{|\kappa_1|(\kappa_1)_v^2\sqrt{E}}{|(\kappa_1)_u|(\kappa_1-\kappa_2)^2G}\leq0.
	\]
	Thus we have the assertion. 
	For $i=2$, we can show similarly. 
\end{proof}

\subsection{Geometric invariants at cuspidal edges of focal surfaces}
We investigate geometric invariants of focal surfaces at cuspidal edges. 
{Let $f\colon U\to\R^3$ be a regular surface without umbilic points.
Assume that $C_i$ ($i=1$ or $2$) is a cuspidal edge at $p\in U$. 
Then we consider the limiting normal curvature $\kappa_\nu^{C_i}$ and 
the singular curvature $\kappa_s^{C_i}$ of $C_i$ at $p$. 
In particular, we shall give the proof of Theorem \ref{thm:ks-c1}.}

{
Using a curvature line coordinate system $(u,v)$ for $f$, 
the limiting normal curvature and the singular curvature for $C_i$ 
obtained as follows.
\begin{proposition}\label{prop:curvatures-C}
	Let $f\colon U\to\R^3$ be a regular surface without umbilic points. 
	Assume that $f$ is parametrized by a curvature line coordinate system $(u,v)$. 
	\begin{enumerate}[label={\rm(\arabic*)}]
		\item If $C_1$ is a cuspidal edge at $p\in U$, 
		then the limiting normal curvature $\kappa_\nu^{C_1}$ and 
		the singular curvature $\kappa_s^{C_1}$ can be written as 
		\begin{equation}\label{eq:kskn-C1}
			\begin{aligned}
				\kappa_\nu^{C_1}&=
				-\frac{\kappa_1^3(\kappa_2)_uG}
				{\sqrt{E}((\kappa_1-\kappa_2)^2\kappa_1^2G+(\kappa_1)_v^2)},\\
				\kappa_s^{C_1}&=
				-\sgn((\kappa_1)_{uu}\kappa_1(\kappa_1-\kappa_2))
				\frac{\kappa_1^3\sqrt{G}X_1}
				{((\kappa_1-\kappa_2)^2\kappa_1^2G+(\kappa_1)_v^2)^{3/2}}
			\end{aligned}
		\end{equation}
		at $p$, where $X_1$ is 
		\[
		X_1=(\kappa_1-\kappa_2)\left(
		(\kappa_1-\kappa_2)\kappa_1\kappa_2G-\frac{\det\hess(\kappa_1)}{(\kappa_1)_{uu}}\right)
		+(\kappa_1)_v(2(\kappa_1)_v+\Gamma^{2}_{22}(\kappa_1-\kappa_2)-(\kappa_2)_v).
		\]
		\item If $C_2$ is a cuspidal edge at $p\in U$, 
		then the limiting normal curvature $\kappa_\nu^{C_2}$ and 
		the singular curvature $\kappa_s^{C_2}$ can be written as 
		\begin{equation}\label{eq:kskn-C2}
			\begin{aligned}
				\kappa_\nu^{C_2}&=
				-\frac{\kappa_2^3(\kappa_1)_vE}
				{\sqrt{G}((\kappa_1-\kappa_2)^2\kappa_2^2E+(\kappa_2)_u^2)},\\
				\kappa_s^{C_2}&=
				-\sgn((\kappa_2)_{vv}\kappa_2(\kappa_2-\kappa_1))
				\frac{\kappa_2^3\sqrt{E}X_2}
				{((\kappa_1-\kappa_2)^2\kappa_2^2E+(\kappa_2)_u^2)^{3/2}}
			\end{aligned}
		\end{equation}
		at $p$, where $X_2$ is 
		\[
		X_2=(\kappa_2-\kappa_1)\left(
		(\kappa_2-\kappa_1)\kappa_1\kappa_2E-\frac{\det\hess(\kappa_2)}{(\kappa_2)_{vv}}\right)
		+(\kappa_2)_u(2(\kappa_2)_u+\Gamma^{1}_{11}(\kappa_2-\kappa_1)-(\kappa_1)_u).
		\]
	\end{enumerate}
\end{proposition}
\begin{proof}
	We show the case of $C_1$. 
	By \eqref{eq:lam-C} and 
	Proposition \ref{prop:sing-C}, $S(C_1)=((\kappa_1)_u)^{-1}(0)$ and $(\kappa_1)_{uu}(p)\neq0$ hold. 
	Thus there exists a regular curve $\gamma(v)=(u(v),v)$ ($|v-v_0|<\eps$) 
	with $\gamma(v_0)=p$ such that 
	$(\kappa_1)_u(u(v),v)=0$ by the implicit function theorem. 
	The curve $\gamma(v)$ is a singular curve for $C_1$, 
	and the curve 
	\begin{equation}\label{eq:hat-gamma}
		\hat{\gamma}(v)=C_1(\gamma(v))=C_1(u(v),v)
	\end{equation}
	gives the singular locus of $C_1$ near $p$. 
	By \eqref{eq:diff-C} and $(\kappa_1)_u(u(v),v)=0$, 
	it holds that 
	\begin{equation}\label{eq:d-gammas}
		\begin{aligned}
			\hat{\gamma}'(v)&=(C_1)_v(\gamma(v))
			=\left(1-\frac{\kappa_2(\gamma(v))}{\kappa_1(\gamma(v))}\right)f_v(\gamma(v))
			-\frac{(\kappa_1)_v(\gamma(v))}{\kappa_1(\gamma(v))^2}\nu(\gamma(v)),\\
			\hat{\gamma}''(v)&=
			(C_1)_{uv}(\gamma(v))u'(v)+(C_1)_{vv}(\gamma(v)),
		\end{aligned}
	\end{equation}	
	where $'=d/dv$. 
	By \eqref{eq:gauss} and \eqref{eq:diff-C}, we have 
	\begin{equation}\label{eq:dd-C}
		\begin{aligned}
			(C_1)_{uv}&=\left(\frac{2(\kappa_1)_u(\kappa_1)_v}{\kappa_1^3}
			-\frac{(\kappa_1)_{uv}}{\kappa_1^2}\right)\nu
			+\frac{\kappa_2(\kappa_1)_u}{\kappa_1^2}f_v
			=-\frac{(\kappa_1)_{uv}}{\kappa_1^2}\nu,\\
			(C_1)_{vv}&=\left(1-\frac{\kappa_2}{\kappa_1}\right)f_{vv}
			+\left(\frac{2\kappa_2(\kappa_1)_v}{\kappa_1^2}-\frac{(\kappa_2)_{v}}{\kappa_1}\right)f_v
			+\left(\frac{2(\kappa_1)_v^2}{\kappa_1^3}-\frac{(\kappa_1)_{vv}}{\kappa_1^2}\right)\nu\\
			&=\left(1-\frac{\kappa_2}{\kappa_1}\right)\Gamma^{1}_{22}f_u
			+\left(\frac{2\kappa_2(\kappa_1)_v}{\kappa_1^2}-\frac{(\kappa_2)_{v}}{\kappa_1}
			+\left(1-\frac{\kappa_2}{\kappa_1}\right)\Gamma^{2}_{22}\right)f_v
			+\left(\frac{2(\kappa_1)_v^2}{\kappa_1^3}-\frac{(\kappa_1)_{vv}}{\kappa_1^2}
			+\left(1-\frac{\kappa_2}{\kappa_1}\right)N\right)\nu
		\end{aligned}
	\end{equation}
	along $\gamma$. 
	On the other hand, differentiating $(\kappa_1)_{u}(u(v),v)=0$ by $v$, 
	we get 
	\[
	(\kappa_1)_{uu}(\gamma(v))u'(v)+(\kappa_1)_{uv}(\gamma(v))=0.	
	\]
	Thus we see that 
	\begin{equation}\label{eq:du}
		u'(v)=-\frac{(\kappa_1)_{uv}(\gamma(v))}{(\kappa_1)_{uu}(\gamma(v))}
	\end{equation}
	near $p$. 
	By \eqref{eq:d-gammas}, \eqref{eq:dd-C} and \eqref{eq:du}, 
	the second derivative of $\hat{\gamma}$ is 
	\begin{equation}\label{eq:sec-dgamma}
		\begin{aligned}
			\hat{\gamma}''(v)
			&=\left(1-\frac{\kappa_2(\gamma(v))}{\kappa_1(\gamma(v))}\right)
			\Gamma^{1}_{22}(\gamma(v))f_u(\gamma(v))\\
			&\quad+\left(\frac{2\kappa_2(\gamma(v))(\kappa_1)_v(\gamma(v))}{\kappa_1(\gamma(v))^2}
			-\frac{(\kappa_2)_{v}(\gamma(v))}{\kappa_1(\gamma(v))}
			+\left(1-\frac{\kappa_2(\gamma(v))}{\kappa_1(\gamma(v))}\right)\Gamma^{2}_{22}(\gamma(v))\right)
			f_v(\gamma(v))\\
			&\quad\quad+\left(-\frac{\det\hess(\kappa_1)(\gamma(v))}{\kappa_1(\gamma(v))^2(\kappa_1)_{uu}(\gamma(v))}
			+\frac{2(\kappa_1)_v(\gamma(v))^2}{\kappa_1(\gamma(v))^3}
			+\left(1-\frac{\kappa_2(\gamma(v))}{\kappa_1(\gamma(v))}\right)N(\gamma(v))	
			\right)\nu(\gamma(v)).
		\end{aligned}
	\end{equation}
	Since $\be_1$ is a unit normal vector to $C_1$, 
	the limiting normal curvature of $C_1$ at $p$ is 
	\[
	\kappa_\nu^{C_1}(p)=\frac{\inner{\hat{\gamma}''(v_0)}{\be_1(p)}}{|\hat{\gamma}'(v_0)|^2}
	\]
	by \eqref{eq:invariants}. 
	By \eqref{eq:gauss}, \eqref{eq:d-gammas}, \eqref{eq:sec-dgamma} 
	and Lemma \ref{lem:d-kappa}, we have 
	\[
	\begin{aligned}
		|\hat{\gamma}'(v_0)|^2&
		=\frac{1}{\kappa_1(p)^4}
		\left((\kappa_1(p)-\kappa_2(p))^2\kappa_1(p)^2G(p)+(\kappa_1)_v(p)^2\right),\\
		\inner{\hat{\gamma}''(v_0)}{\be_1(p)}&=
		\frac{\kappa_1(p)-\kappa_2(p)}{\kappa_1(p)}\Gamma^{1}_{22}(p)\sqrt{E(p)}
		=-\frac{\kappa_1(p)-\kappa_2(p)}{\kappa_1(p)}\frac{G_u}{2\sqrt{E(p)}}
		=-\frac{(\kappa_2)_u(p)G(p)}{\kappa_1(p)\sqrt{E(p)}}.
	\end{aligned}
	\]
	Thus we obtain the assertion for $\kappa_\nu^{C_1}$. 
	
	We next consider $\kappa_s^{C_1}$. 
	The singular curvature $\kappa_s^{C_1}$ at $p$ is 
	\[
	\kappa_s^{C_1}(p)
	=\sgn(\det(\gamma',\eta^{C_1})(v_0)\cdot\eta^{C_1}\lambda^{C_1}(p))
	\frac{\det(\hat{\gamma}'(v_0),\hat{\gamma}''(v_0),\be_1(p))}
	{|\hat{\gamma}'(v_0)|^3}.
	\]	
	By the scalar triple product, 
	it holds that 
	\[
	\det(\hat{\gamma}'(v_0),\hat{\gamma}''(v_0),\be_1(p))
	=\inner{\be_1(p)\times\hat{\gamma}'(v_0)}{\hat{\gamma}''(v_0)}.
	\]
	By \eqref{eq:d-gammas}, we see that 
	\[
	\be_1(p)\times\hat{\gamma}'(v_0)
	=\frac{(\kappa_1)_v(p)}{\kappa_1(p)^2}\be_2(p)
	+\frac{\kappa_1(p)-\kappa_2(p)}{\kappa_1(p)}\sqrt{G(p)}\nu(p).
	\]
	Thus by \eqref{eq:sec-dgamma}, we have 
	\[
	\det(\hat{\gamma}'(v_0),\hat{\gamma}''(v_0),\be_1(p))
	=\inner{\be_1(p)\times \hat{\gamma}'(v_0)}{\hat{\gamma}''(v_0)}
	=\frac{\sqrt{G}(p)}{\kappa_1(p)^3}X_1(p).
	\]
	On the other hand, 
	since $\eta^{C_1}=\partial_u$, we get 
	\[
	\eta^{C_1}\lambda^{C_1}(p)=\frac{(\kappa_1)_{uu}(p)(\kappa_1(p)-\kappa_2(p))}{\kappa_1(p)^3}|f_v(p)|,\quad 
	\det(\gamma',\eta)(v_0)=-1
	\]
	by \eqref{eq:lam-C}. 
	Hence it follows that 
	\[
	\sgn(\det(\gamma',\eta)(v_0)\cdot\eta^{C_1}\lambda^{C_1}(p))
	=-\sgn((\kappa_1)_{uu}(p)\kappa_1(p)(\kappa_1(p)-\kappa_2(p))).
	\]
	Thus we have the conclusion for $\kappa_s^{C_1}$. 
	
	For $C_2$, one can show in a similar manner. 
\end{proof}
}

By using Proposition \ref{prop:curvatures-C}, 
we show Theorem \ref{thm:ks-c1}. 
\begin{proof}[Proof of Theorem \ref{thm:ks-c1}]
We show the case of $C_1$. 
By Proposition \ref{prop:kn-para}, 
$\kappa_2(p)=0$. 
Let us take a curvature line coordinate system $(u,v)$. 
Then by Proposition \ref{prop:lips/beaks}, 
it holds that $(\kappa_1)_u(p)=(\kappa_1)_v(p)=0$ 
and 
$\det\hess(\kappa_1)(p)>0$ (resp. $\det\hess(\kappa_1)(p)<0$) 
when $f^t$ is a cuspidal lips (resp. a cuspidal beaks) at $p$. 
On the other hand, by Proposition \ref{prop:curvatures-C}, 
the singular curvature $\kappa_s^{C_1}(p)$ is given by 
\[
\kappa_s^{C_1}(p)
=\sgn((\kappa_1)_{uu}(p))\frac{\det\hess(\kappa_1)(p)}{(\kappa_1)_{uu}(p)\kappa_1(p)^2G(p)}.
\]
Therefore $\sgn(\kappa_s^{C_1}(p))=\sgn(\det\hess(\kappa_1)(p))$ holds. 
Hence we get the assertion. 
\end{proof}
\begin{example}
Let $f\colon\R^2\to\R^3$ be a $C^\infty$ map given by 
\[
f(u,v)=\left(u,v,\frac{1}{2}u^2+u^4+u^3v\right).
\]
Then principal curvatures satisfy 
\[
\kappa_1(0,0)=1,\quad
\kappa_2(0,0)=0.
\]
Moreover, we see that the parallel surface $f^1=f+\nu$ 
is a cuspidal beaks at $(0,0)$ (see Figure \ref{fig:ex-cbk}). 
We find that the limiting normal curvature of $f^1$ vanishes at $(0,0)$. 
In addition, the singular curvature $\kappa_s^{C_1}$ of $C_1$ is 
$\kappa_s^{C_1}(0)=-12/7<0$. 

\begin{figure}[h]
	\centering
	\includegraphics[width=0.3\linewidth]{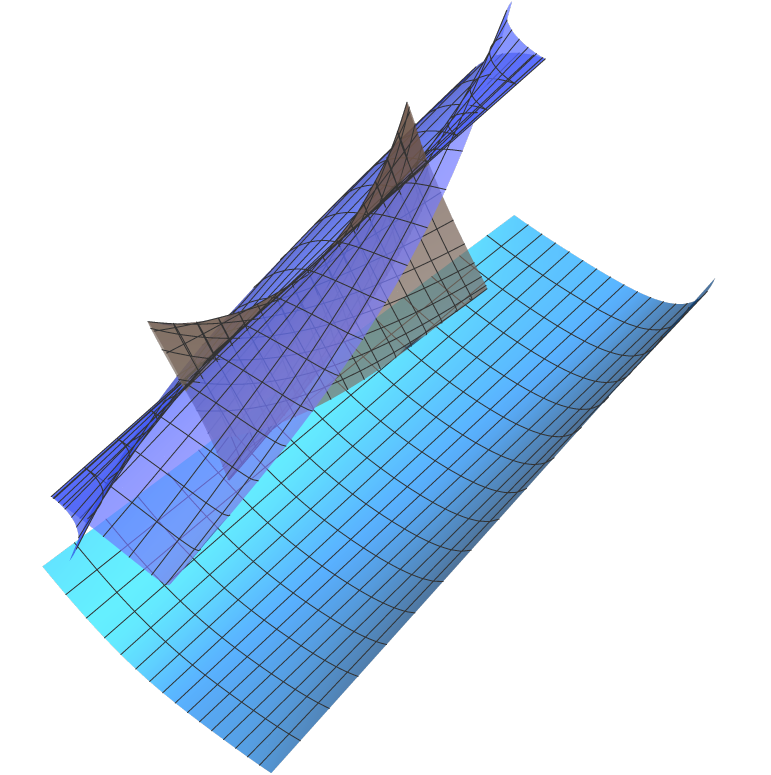}
	\caption{The images of $f$ (blue), $f^1$ (purple) and $C_1$ (gray).}
	\label{fig:ex-cbk}
\end{figure}
\end{example}

\begin{example}
Let $f\colon \R^2\to\R^3$ be a regular surface 
given by 
\[
f(u,v)=\left(u,v,\frac{1}{2}u^2+uv^2+u^4\right).
\]
Then 
\[
\nu(u,v)=\frac{(-u-v^2-4u^3,-2uv,1)}{\sqrt{1+4u^2v^2+(u+v^2+4u^3)^2}}
\]
is a unit normal vector to $f$. 
By direct calculation, we see that $f^1=f+\nu$ is a cuspidal lips at $(0,0)$. 
Moreover, it can be observed that the singular curvature $\kappa_s^{C_1}$ of 
the focal surface $C_1$ is $\kappa_s^{C_1}(0)=8>0$ at $(0,0)$ 
(see Figure \ref{fig:ex-clp}). 

\begin{figure}[h]
	\centering
	\includegraphics[width=0.3\linewidth]{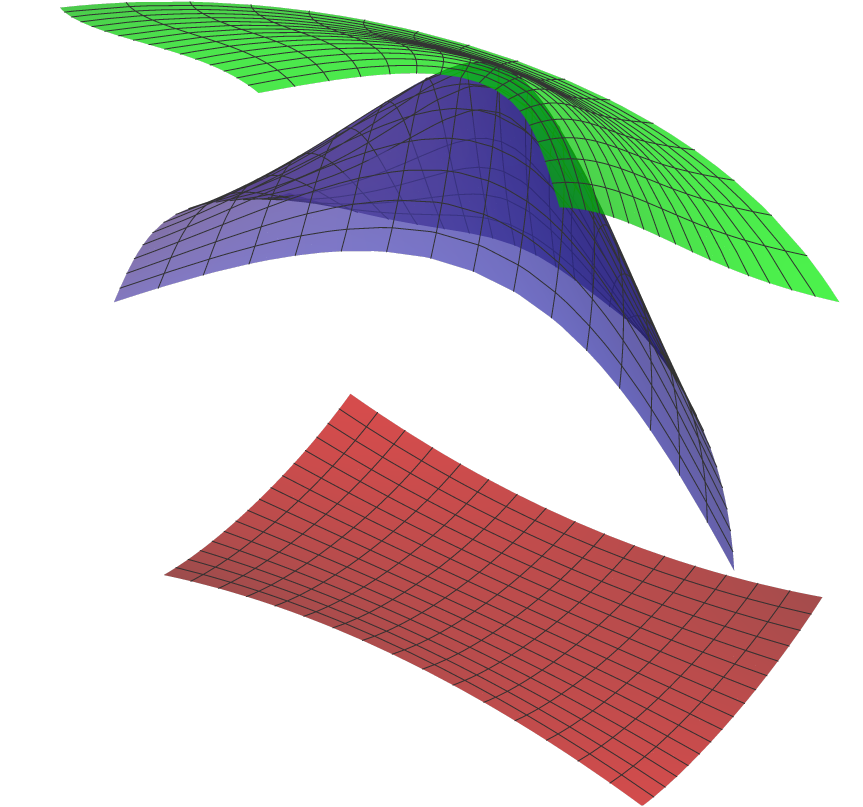}
	\caption{The images of $f$ (red), $f^1$ (green) and $C_1$ (purple).}
	\label{fig:ex-clp}
\end{figure}
\end{example}

{
\begin{example}\label{ex:counter}
	Let $f\colon \R^2\to\R^3$ be a regular surface given by 
	\[
	f(u,v)=\left(u,v,\frac{1}{2}u^2-v^2+u^4+u^3v\right).
	\]
	Then the unit normal vector field $\nu$ is 
	\[
	\nu(u,v)=\frac{(-u(1+4u^2+3uv),-u^3+2v,1)}
	{\sqrt{1+(u^3+2v)^2+u^2(1+4u^2+3uv)^2}}.
	\]
	Principal curvatures $\kappa_1,\kappa_2$ of $f$ can be chosen as  
	\[
	\kappa_1(0,0)=-2,\quad\kappa_2(0,0)=1\neq0. 
	\]
	Then a parallel surface $f^{-1/2}=f-\nu/2$ has a cuspidal lips at $(0,0)$ 
	and $\kappa_\nu^{t}(0)\neq0$. 
	Moreover, the singular curvature of $C_1$ at the origin is calculated as 
	$\kappa_s^{C_1}(0)=-9/4<0$ 
	(see Figure \ref{fig:cex}). 
	This gives a counterexample of Theorem \ref{thm:ks-c1}. 
	\begin{figure}[h]
		\centering
		\begin{minipage}[b]{0.3\hsize}
			\centering
			\includegraphics[width=0.65\linewidth]{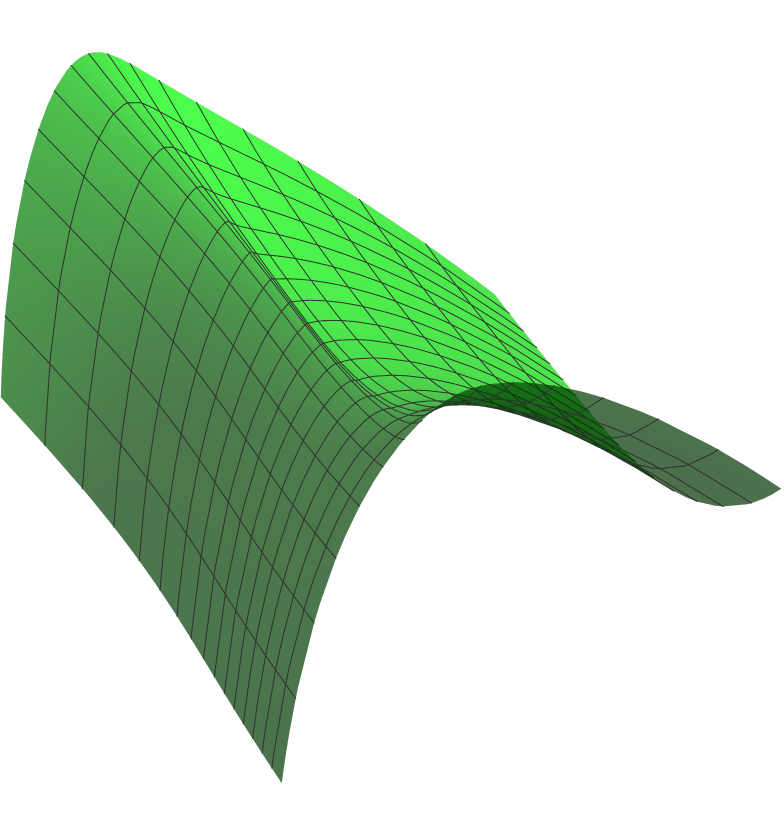}
		\end{minipage}
		\begin{minipage}[b]{0.3\hsize}
			\centering
			\includegraphics[width=0.65\linewidth]{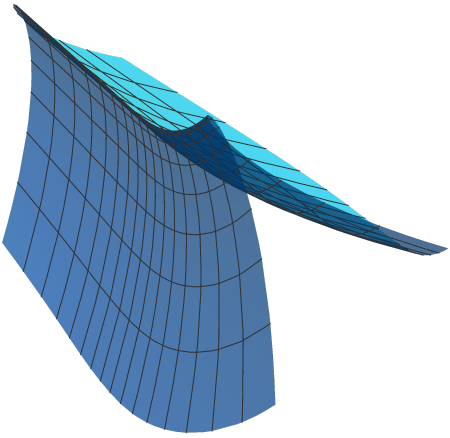}
		\end{minipage}
		\begin{minipage}[b]{0.3\hsize}
			\centering
			\includegraphics[width=0.65\linewidth]{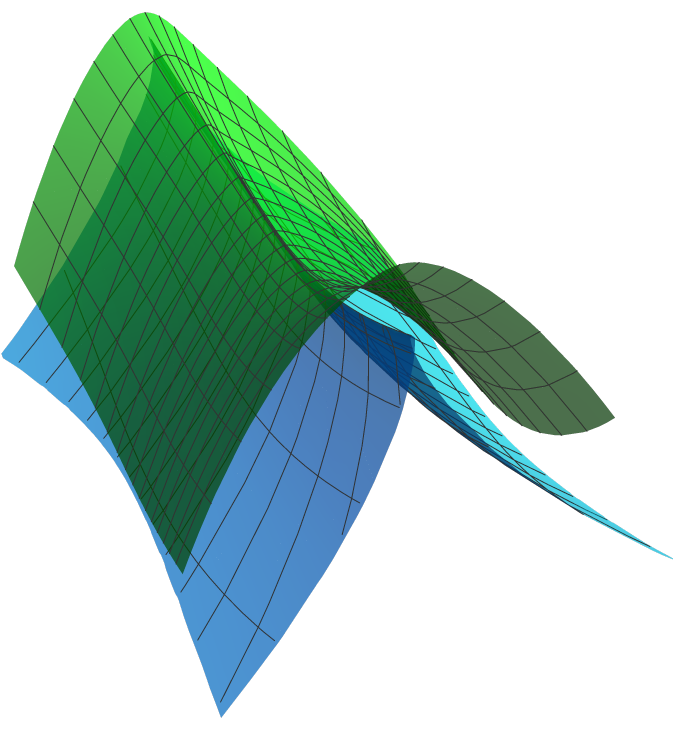}
		\end{minipage}
		\caption{The image of $f^{-1/2}$ (left), 
			$C_1$ (middle) and both of them (right). }
		\label{fig:cex}
	\end{figure}
\end{example}

If the Gaussian curvature of $f^t$ is bounded near 
its singular point $p$, then $\kappa_\nu^t$ vanishes at $p$ 
(see Appendix \ref{sec:appendix}). 
Thus we get the following by Theorem \ref{thm:ks-c1} immediately. 
\begin{corollary}
	Let $f\colon U\to\R^3$ be a regular surface 
	whose parallel surface $f^t$ is a cuspidal lips \textup{(}resp. cuspidal beaks\textup{)} 
	at $p\in U$, where $t=1/\kappa_i(p)$ \textup{(}$i=1$ or $2$\textup{)}. 
	If the Gaussian curvature $K^t$ of $f^t$ is bounded near $p$, 
	then $\kappa_s^{C_i}$ for $C_i$ 
	is positive \textup{(}resp. negative\textup{)} at $p$.
\end{corollary}
\begin{proof}
	Let us take a curvature line coordinate system $(u,v)$. 
	We then consider the case of $i=1$. 
	By the assumption, it follows that $\kappa_1^{-1}(\kappa_1(p))\subset \kappa_2^{-1}(0)$ 
	(see Appendix \ref{sec:appendix}). 
	Thus we have $\kappa_2(p)=0$, 
	and hence $\kappa_\nu^t(p)=0$. 
	Therefore we have the assertion by Theorem \ref{thm:ks-c1}. 
	For $i=2$, one can show in the similar manner. 
\end{proof}
}

When the limiting normal curvature of the parallel surface does not vanish, 
we have the following.

\begin{proposition}
	Let $f\colon U\to\R^3$ be a regular surface without umbilic point. 
	Assume that $f$ is parametrized by a curvature line coordinate system $(u,v)$. 
	\begin{enumerate}
		\item If the parallel surface $f^t$ with $t=1/\kappa_1(p)$ 
		is a cuspidal lips \textup{(}resp. a cuspidal beaks\textup{)} 
		and $\sgn(\kappa_\nu^t(\kappa_1)_{uu})=-1$ \textup{(}resp. $1$\textup{)} at $p$, 
		then the singular curvature $\kappa_s^{C_1}$ of $C_1$ is positive 
		\textup{(}resp. negative\textup{)} at $p$.
		\item If the parallel surface $f^t$ with $t=1/\kappa_2(p)$ 
		is a cuspidal lips \textup{(}resp. a cuspidal beaks\textup{)} 
		and $\sgn(\kappa_\nu^t(\kappa_2)_{vv})=-1$ \textup{(}resp. $1$\textup{)} at $p$, 
		then the singular curvature $\kappa_s^{C_2}$ of $C_2$ is positive 
		\textup{(}resp. negative\textup{)} at $p$.
	\end{enumerate}
\end{proposition}
\begin{proof}
	We show the first assertion. 
	We assume that $f^t$ is a cuspidal lips or cuspidal beaks. 
	Then the singular curvature $\kappa_s^{C_1}$ of $C_1$ is 
	\[
	\begin{aligned}
		\kappa_s^{C_1}&=-\sgn((\kappa_1)_{uu}\kappa_1(\kappa_1-\kappa_2))
		\frac{\kappa_1^3X_1\sqrt{G}}{|\kappa_1-\kappa_2|^3|\kappa_1|^3G^{3/2}}\\
		&=-\frac{(\kappa_1)_{uu}\kappa_1(\kappa_1-\kappa_2)}
		{|(\kappa_1)_{uu}||\kappa_1||\kappa_1-\kappa_2|}
		\frac{\kappa_1^3(\kappa_1-\kappa_2)\sqrt{G}}{|\kappa_1-\kappa_2|^3|\kappa_1|^3G^{3/2}}
		\left(\kappa_1\kappa_2(\kappa_1-\kappa_2)G-\frac{\det\hess(\kappa_1)}{(\kappa_1)_{uu}}\right)\\
		&=-\frac{\kappa_1\kappa_2(\kappa_1-\kappa_2)(\kappa_1)_{uu}G-\det\hess(\kappa_1)}
		{(\kappa_1-\kappa_2)^2|(\kappa_1)_{uu}|G}
	\end{aligned}
	\]
	at $p$ by Proposition \ref{prop:curvatures-C}. 
	On the other hand, the limiting normal curvature $\kappa_\nu^t$ of $f^t$ is 
	\[
	\kappa_\nu^t=\frac{\kappa_1\kappa_2}{\kappa_1-\kappa_2}
	=\frac{\kappa_1\kappa_2(\kappa_1-\kappa_2)}{(\kappa_1-\kappa_2)^2}
	\]
	at $p$ by Proposition \ref{prop:kn-para}. 
	Thus $\sgn(\kappa_\nu^t)=\sgn(\kappa_1\kappa_2(\kappa_1-\kappa_2))$. 
	Since $\det\hess(\kappa_1)>0$ (resp. $<0$) at $p$ 
	when $f^t$ is a cuspidal lips (resp. cuspidal beaks) at $p$, 
	we have the conclusion. 
	
	For the second case, we can show in the similar way. 
\end{proof}

\appendix
{
\section{Gaussian curvature of parallels and focal surfaces}\label{sec:appendix}
Let $f\colon U\to\R^3$ be a regular surface 
and $\nu$ its unit normal vector field of $f$. 
Then the Gaussian curvature $K^t$ of the parallel surface $f^t=f+t\nu$ is 
\[
K^t=\frac{\kappa_1\kappa_2}{(1-t\kappa_1)(1-t\kappa_2)}
\]
(cf. \cite{eisen}). 
Thus when $t=1/\kappa_1(p)$ (resp. $t=1/\kappa_2(p)$), 
$K^t$ is bounded near $p$ if and only if 
$\kappa_1^{-1}(\kappa_1(p))\subset \kappa_2^{-1}(0)$ 
(resp. $\kappa_2^{-1}(\kappa_2(p))\subset \kappa_1^{-1}(0)$) (cf. \cite{suy-front,msuy}). 
Moreover, the Gaussian curvature $K^{C_i}$ ($i=1,2$) of $C_i$ can be calculated as 
\begin{equation}\label{eq:KC}
	K^{C_i}
	=\begin{dcases}
		-\frac{\kappa_1^4(\kappa_2)_{u}}{(\kappa_1)_u(\kappa_1-\kappa_2)^2} & (i=1),\\
		-\frac{\kappa_2^4(\kappa_1)_{v}}{(\kappa_2)_v(\kappa_1-\kappa_2)^2} & (i=2)
	\end{dcases}
\end{equation}
(\cite{morris,eisen}). 
Thus the Gaussian curvature $K^{C_1}$ of $C_1$ 
is bounded near a singular point $p$ 
if and only if $((\kappa_1)_u)^{-1}(0)\subset((\kappa_2)_{u})^{-1}(0)$. 
In particular, if $K^{C_1}$ is bounded and takes a non-zero value at $p$, 
then $((\kappa_1)_u)^{-1}(0)=((\kappa_2)_{u})^{-1}(0)$.

Assume that the Gaussian curvature of $f$ is a non-zero constant $c$. 
Then it holds that $\kappa_2=c/\kappa_1$. 
Thus we see that $((\kappa_2)_u)^{-1}(0)=((\kappa_1)_u)^{-1}(0)$. 
This implies that $K^{C_1}$ is bounded near the set of singular points of $C_1$. 
Moreover, by \eqref{eq:KC}, the Gaussian curvature $K^{C_1}$ of $C_1$ is given by 
\begin{equation}\label{eq:KC2}
	K^{C_1}=\frac{c\kappa_1^4}{(c-\kappa_1^2)^2}.
\end{equation}
Therefore if the surface $f$ is a constant positive (resp. negative) Gaussian curvature, 
then the Gaussian curvature $K^{C_1}$ of the caustic takes positive (resp. negative) values. 
}





%
%






\bibliographystyle{amsplain}

\bibliography{references}



\end{document}